\documentclass[12pt]{amsart}
\usepackage{tikz,array,verbatim}
\usepackage{amsfonts, amsmath, latexsym, epsfig, caption}
\usepackage{amssymb, color}

\usepackage{epsf}
\usepackage{array}
\usepackage{ragged2e}
\usepackage{hyperref}
\usepackage[margin=1in]{geometry}
\usepackage{longtable}

\newtheorem{theorem}{Theorem}[section]

\theoremstyle{definition}

\begin{document}

\title[String C-groups of 2-power order]{String C-groups of 2-power order project onto a common string C-group}

\author[D.-D. Hou]{Dong-Dong Hou}
\address{Dong-Dong Hou, Shanxi Key Laboratory of Cryptography and Data Security, Department of Mathematics, Shanxi Normal University, Taiyuan, 100044, P.R. China}
\email{holderhandsome@bjtu.edu.cn}

\author[E. Schulte]{\\ Egon Schulte}
\address{Egon Schulte, Department of Mathematics, Northeastern University, Boston, MA, 02115, USA}
\email{e.schulte@northeastern.edu}

\date{ \today }

\maketitle

\begin{abstract}
String C-groups are precisely the automorphism groups of abstract regular polytopes.
A certain regular $d$-polytope $\mathcal{C}_d$ with an automorphism group of order $2^{2d-1}$, discovered by Conder and shown to have the smallest number of flags among all regular $d$-polytopes of high ranks, also has the important extremal property to be the unique minimal $d$-polytope, with respect to combinatorial covering, among all finite regular $d$-polytopes with 2-power automorphism groups. In other words, the automorphism group of $\mathcal{C}_d$ is a quotient group of every finite string C-group of rank $d$ and 2-power order; and every finite regular $d$-polytope with an automorphism groups of 2-power order covers $\mathcal{C}_d$. The existence of a unique minimal element among string C-groups
of 2-power order and given rank is remarkable in itself.
\bigskip

\medskip

\end{abstract}

\section{Introduction}
\label{intro}

String C-groups arise naturally as automorphism groups of abstract regular polytopes. Abstract polytopes are ranked combinatorial structures with distinctive geometric, algebraic, or topological characteristics. They have attracted considerable attention over the past few decades (McMullen \& Schulte~\cite{McMSch2002}, Pellicer~\cite{Pel2025}). A large body of research in this area concerns the classification and analysis of regular, chiral, or other highly symmetric abstract polytopes, usually with a focus on the polytopes themselves.

A separate direction of research begins with the groups and describes the regular, chiral, or other symmetric polytopes whose automorphism group is of a given type; for example, a symmetric group~\cite{CFL2024}, an alternating group~\cite{CFLM2017,CHO2024}, or another almost simple group such as a projective linear group~\cite{BroLee2016,BroVic2010,LeeSch2007}.

The present paper investigates string $C$-groups whose order is a power of 2. These groups are precisely the automorphism groups of regular polytopes whose number of flags is a power of 2. Our work continues a line of inquiry, initiated by an open problem suggested in~\cite{SchWei2006} but properly started only by Hou, Feng, Leemans and Qu~\cite{HFL2019a,HFL2019b,HFL2020,HFL2025,HFLQ2025}, as well as Gomi, Loyola and De Las Penas~\cite{GLD2018} for special cases, and further developed by Cunningham, Feng, Hou and Schulte~\cite{CFHS2026}.

Our goal is to establish the following two theorems, which are essentially equivalent. Let $\mathcal{C}_d$ denote the small regular $d$-polytope of type $\{4,\ldots,4\}$ with an automorphism group of order $2^{2d-1}$ discovered in Conder~\cite{Con2013}.  Our theorems highlight important extremal properties of $\mathcal{C}_d$.

\begin{theorem}
\label{mainthm1}
Every finite non-degenerate \footnote{A string C-group is non-degenerate if its Schl\"afli type does not contain 2's.} string C-group of 2-power order and rank $d\geq 2$ projects onto the automorphism group of the regular $d$-polytope $\mathcal{C}_d$ (under a homomorphism mapping distinguished generators to distinguished generators).
\end{theorem}

\begin{theorem}
\label{mainthm2}
Every finite non-degenerate regular polytope $\mathcal{P}$ of rank $d\geq 2$ with an automorphism group of 2-power order is a covering of the regular $d$-polytope $\mathcal{C}_d$.
\end{theorem}

These two theorems prove the remarkable fact that among all regular $d$-polytopes with 2-power automorphism groups there is a unique minimal $d$-polytope, $\mathcal{C}_d$, with respect to combinatorial covering. This regular $d$-polytope $\mathcal{C}_d$, which is minimal (with respect to covering) for the class of regular $d$-polytopes with 2-power automorphism groups, is further distinguished by being the smallest in group size (and number of flags) among {\it all} regular $d$-polytopes, with some small exceptions~\cite{Con2013}, irrespective of the type of group.

The paper is organized as follows. Section~\ref{bano} reviews basic properties of string C-groups and regular polytopes, and Section~\ref{secsmallest} describes the regular $d$-polytopes $\mathcal{C}_d$. The final Section~\ref{cdminimality} investigates string C-group of 2-power order and proves Theorems~\ref{mainthm1} and \ref{mainthm2}.

\section{Basic Notions}
\label{bano}

In this section we briefly review some background about string C-groups and abstract regular polytopes. For more details the reader is referred to~\cite[Chs.\,2,\,3]{McMSch2002}.

String C-groups of rank $d$ are precisely the automorphism groups are abstract regular polytopes of rank $d$.  An (\emph{abstract\/}) \textit{polytope of rank\/} $d$ is a ranked partially ordered set $\mathcal{P}$ consisting of \textit{faces} of ranks $-1,0, \ldots, d$. The faces of rank $j$ are called \textit{$j$-faces\/}, and the faces of ranks $0$, $1$ or $d-1$ are called \textit{vertices\/}, \textit{edges\/} or \textit{facets\/}, respectively. There are a unique least face $F_{-1}$ (of rank $-1$) and a unique greatest face $F_d$ (of rank $d$) in $\mathcal{P}$. A \textit{flag\/} is a maximal totally ordered subset (chain) of $\mathcal{P}$. Each flag (maximal totally ordered subset) has one face for each rank, including $F_{-1}$ and $F_{d}$. Two flags are said to be \textit{adjacent} if they differ in exactly one face; they are called $j$-\textit{adjacent} if this is a $j$-face. A polytope $\mathcal{P}$ is \textit{strongly flag-connected}, in the sense that for any two flags $\Phi$ and $\Psi$ there exists a finite sequence of successively adjacent flags, each containing $\Phi \cap \Psi$. Finally, $\mathcal{P}$ has the \textit{diamond property\/}:\ whenever $F \leq G$, with $F$ a $(j-1)$-face and $G$ a $(j+1)$-face for some~$j$, there are exactly two $j$-faces $H$ with $F \leq H \leq G$.  Abstract 3-polytopes are also called (\emph{abstract}) \emph{polyhedra} and can be realized topologically as maps on closed surfaces.

If $F$ and $G$ are faces with $F \leq G$, we call
\[G/F := \{ H \in \mathcal{P}\, | \, F \leq H \leq G \}\]
a \textit{section} of $\mathcal{P}$ and note that this is a $(k-j-1)$-polytope in its own right if $F$ has rank $j$ and $G$ has rank $k$. A face $F$ naturally gives rise to two sections, namely, $F/F_{-1}$ usually identified with $F$, and $F_{d}/F$ called the \textit{co-face\/} of $\mathcal{P}$ at $F$, or \emph{vertex-figure\/} of $\mathcal{P}$ at $F$ if $F$ is a vertex.

A polytope $\mathcal{P}$ is called \emph{regular\/} if its automorphism group is transitive on the flags. Every regular polytope has regular facets and vertex-figures, and more generally, sections.  A regular $d$-polytope $\mathcal{P}$ is said to be of (Schl\"afli) type $\{p_1,\ldots,p_{d-1}\}$ if, for each $i=1,\ldots,d-1$, each section lying between an $(i-2)$-face and an incident $(i+1)$-face is isomorphic to a $p_i$-gon (possibly $p_{i}=\infty$).

The automorphism group of a regular $d$-polytope $\mathcal{P}$ is generated by involutions $\rho_0,\ldots,\rho_{d-1}$, where each $\rho_i$ maps a (fixed) \textit{base\/} flag $\Phi$ to its unique $i$-adjacent flag. These generators satisfy (at least) the standard Coxeter relations for string diagrams,
\begin{equation}
\label{Coxdia}
\centering
\begin{picture}(150,15)
\put(100,0){
\multiput(60,0)(45,0){2}{\circle*{4}}
\put(115,0){$,$}
\multiput(-130,0)(45,0){3}{\circle*{4}}
\put(40,0){\line(1,0){20}}
\put(-85,0){\line(1,0){20}}
\put(-130,0){\line(1,0){45}}
\put(-85,0){\line(1,0){45}}
\put(-40,0){\line(1,0){20}}
\put(60,0){\line(1,0){45}}
\put(-40,0){\line(1,0){20}}
\put(49,9){\scriptsize $\rho_{d-2}$}
\put(93,9){\scriptsize $\rho_{d-1}$}
\put(78,-8){\scriptsize $p_{d-1}$}
\put(-134,9){\scriptsize $\rho_{0}$}
\put(-89,9){\scriptsize $\rho_{1}$}
\put(-44,9){\scriptsize $\rho_{2}$}
\put(-110,-8){\scriptsize $p_1$}
\put(-65,-8){\scriptsize $p_{2}$}
\put(-4,-0.5){$\ldots\ldots$}}
\end{picture}
\end{equation}
\vskip.1in
\noindent
where the numbers $p_j$ are the entries of the Schl\"afli symbol. In addition, the automorphism group of $\mathcal{P}$ satisfies the following {\em intersection condition\/},
\begin{equation}
\label{intprop}
\langle \rho_i \mid i \in K \rangle \cap \langle \rho_i \mid i \in J \rangle
= \langle \rho_i \mid i \in {K \cap J} \rangle
  \;\; \textrm{ for } K,J \subseteq \{0,1,\ldots,d-1\}.
\end{equation}

Next we recall the concepts of a C-group and string C-group. Then it will be clear that the automorphism group of every regular $d$-polytope is a string C-group of rank $d$.

A group $G$ generated by involutions $\rho_{0},\ldots,\rho_{d-1}$ is called a \textit{C-group} if $G$ has the intersection property~(\ref{intprop}) with respect to these generators. In this case, the number of generators, $d$, is called the \textit{rank} of $G$. We sometimes use the notation $(G,\{\rho_{0},\ldots,\rho_{d-1}\})$ to emphasize the generators of $G$. If $p_{ij}$ denotes the order of the product $\rho_i\rho_j$ in a C-group $G$, then $G$ satisfies (at least) the standard Coxeter relations
\begin{equation}
\label{stringdia}
(\rho_i \rho_j)^{p_{ij}} = 1\quad (i,j=0, \ldots,d-1),
\end{equation}
with $p_{ii}=1$ and $p_{ji} = p_{ij}\geq 2$ for $i\neq j$. In other words, C-groups are quotients of Coxeter groups (with any diagram) that obey the intersection property.

A {\it string C-group\/} is a C-group $G$ whose underlying Coxeter diagram is a string; that is, in addition to being a C-group, the orders of the products must satisfy $p_{ij}=2$ for $|i-j|\geq 2$ (that is, non-adjacent generators must commute).

Every string $C$-group $G$ is the automorphism group of a regular $d$-polytope $\mathcal{P}$ such that the generators $\rho_{0},\ldots,\rho_{d-1}$ of $G$ become the distinguished generators of the automorphism group of $\mathcal{P}$ (\cite[Ch. 2E]{McMSch2002}). The type $\{p_1,\ldots,p_{d-1}\}$ is determined by the branch labels in the underlying string Coxeter diagram. Thus, the string C-groups are precisely the automorphism groups of regular polytopes.

For example, every Coxeter group with a string Coxeter diagram with $d$ nodes and branch labes $p_{1},p_{2},\ldots,p_{d-1}$ is a string C-group of rank $d$, namely the automorphism group of the \textit{universal} regular $d$-polytope $\{p_{1},\ldots,p_{d-1}\}$ (see \cite[Ch. 3D]{McMSch2002}).

A $d$-polytope $\mathcal{P}$ is said to be {\em $(k,l)$-flat\/}, where $0 \leq k < l \leq d-1$, if each of its $k$-faces is incident with each of its $l$-faces. In a $(0,d-1)$-flat polytope,
any vertex is incident with any facet. We simply refer to a $(0,d-1)$-flat polytope as a \textit{flat} polytope.

\section{The smallest regular polytopes}
\label{secsmallest}

In this section, we briefly summarize the results of Conder~\cite{Con2013} on the smallest regular polytopes of a given rank.

For a regular $d$-polytope $\mathcal{P}$ of a given Schl\"afli type $\{p_1,\ldots,p_{d-1}\}$, there is a lower bound for the order of the automorphism group $\Gamma(\mathcal{P})$ given by
\begin{equation}
\label{lowbound}
|\Gamma(\mathcal{P})|\,\geq\, 2p_1p_2\cdot\ldots\cdot p_{d-1}.
\end{equation}
As the order of the automorphism group of a regular polytope coincides with the number of its flags, the term on the right hand side is also a lower bound for the total number of flags in $\mathcal{P}$. Regular polytopes which attain this lower bound are called {\it tight}. Tight regular polytopes have the smallest number of automorphisms (or flags) among all regular polytopes of the given type. In particular, they have $p_1$ vertices, $p_{d-1}$ facets, and $p_{j}p_{j+1}/2$ faces of rank $j$ for $j=1,\ldots, d-2$.

There also exists an {\sl absolute} lower bound for the total number of automorphisms (or flags) of a regular $d$-polytope, irrespective of the Schl\"afli type. This bound only depends on the rank $d$. With a few small exceptions (described in \cite{Con2013}), the smallest regular $d$-polytopes come from a family of tight regular $d$-polytopes $\mathcal{C}_d$ of type $\{4,4,\ldots,4\}$, one for each rank~$d$, and with an automorphism group of order $2\cdot 4^{d-1}=2^{2d-1}$ (or, equivalently, with a total number of flags given by $2\cdot 4^{d-1}=2^{2d-1}$).

The generalized Schl\"afli type of the regular $d$-polytope $\mathcal{C}_d$ is given by
\[\{\{4,4\}_{(2,0)},\{4,4\}_{(2,0)},\ldots,\{4,4\}_{(2,0)}\}.\]
The generalized Schl\"afli type $\{\mathcal{P}_3,\ldots,\mathcal{P}_{d}\}$ of a regular $d$-polytope $\mathcal{P}$ records the structure of its sections of rank 3; that is, if $\{F_{-1},F_0,\ldots,F_d$ is any flag of $\mathcal{P}$, then the section $F_{j}/F_{j-4}$ of $\mathcal{P}$ is isomorphic to the regular polyhedron $\mathcal{P}_j$ for each $j=3,\ldots,d$.
For $\mathcal{C}_d$, all sections of rank 3 happen to be toroidal polyhedra $\{4,4\}_{(2,0)}$, obtained from a $2\times 2$ chessboard by identifying opposites sides~\cite{CoxMos1980}.
The polytope $\mathcal{C}_d$ is very small and has only $4$ vertices, $4$ facets, and 8 faces of each rank $j$ with $j=1,\ldots, d-2$.

The automorphism group $G(\mathcal{C}_d)=\langle\rho_0,\ldots,\rho_{d-1}\rangle$ of $\mathcal{C}_d$ is the quotient of the Coxeter group $[4,4,\ldots,4]=\langle\rho_0,\ldots,\rho_{d-1}\rangle$ (we deliberately keep the same names for the generators, slightly abusing notation), with diagram
\begin{equation}
\label{Coxdia}
\centering
\begin{picture}(150,15)
\put(100,0){
\multiput(60,0)(45,0){2}{\circle*{4}}
\put(115,0){$,$}
\multiput(-130,0)(45,0){3}{\circle*{4}}
\put(40,0){\line(1,0){20}}
\put(-85,0){\line(1,0){20}}
\put(-130,0){\line(1,0){45}}
\put(-85,0){\line(1,0){45}}
\put(-40,0){\line(1,0){20}}
\put(60,0){\line(1,0){45}}
\put(-40,0){\line(1,0){20}}
\put(49,9){\scriptsize $\rho_{d-2}$}
\put(93,9){\scriptsize $\rho_{d-1}$}
\put(78,-8){\scriptsize $4$}
\put(-134,9){\scriptsize $\rho_{0}$}
\put(-89,9){\scriptsize $\rho_{1}$}
\put(-44,9){\scriptsize $\rho_{2}$}
\put(-110,-8){\scriptsize $4$}
\put(-65,-8){\scriptsize $4$}
\put(-4,-0.5){$\ldots\ldots$}}
\end{picture}
\end{equation}
\vskip.1in
\noindent
defined by the extra relations
\begin{equation}
\label{exrelscd}
[(\rho_{i}\rho_{i+1})^2, \rho_j] \,=\, 1 \quad (0\leq i<d-1;\,0\leq j\leq d-1),
\end{equation}
which force the elements $(\rho_i\rho_{i+1})^2$ to lie in the center $Z(G(\mathcal{C}_d))$ of $G(\mathcal{C}_d)$.\smallskip

Define the subgroup $K$ of $Z(G(\mathcal{C}_d))$ by
\[ K:= \langle (\rho_0\rho_1)^2, (\rho_1\rho_2)^2, \ldots, (\rho_{d-3}\rho_{d-2})^2, (\rho_{d-2}\rho_{d-1})^2\rangle, \]
and note that $K$ is abelian and normal in $G(\mathcal{C}_d)$. Then the quotient $G(\mathcal{C}_d)/K$ must also be abelian, since any two of its induced generators $\rho_{i}K$ commute by the definition of $K$. On the other hand, $G(\mathcal{C}_d)$ itself is not abelian. As a 2-group, $G(\mathcal{C}_d)$ is nilpotent. In particular, the upper central series
\[ \{1\} \,\leq\, K\, \leq\, G(\mathcal{C}_d) \]
for $G(\mathcal{C}_d)$ shows that $G(\mathcal{C}_d)$ is of nilpotency class~$2$.

\section{String C-group of order $2^n$ as extensions}
\label{cdminimality}

In this section we exploit Frattini subgroups and agemo subgroups of finite $p$-groups to establish the remarkable fact that among the finite regular $d$-polytopes whose automorphism group are 2-groups there exists a unique minimal element with respect to combinatorial covering, namely the regular $d$-polytope $\mathcal{C}_d$ of \cite{Con2013} described in the previous section.

Recall that the {\em Frattini subgroup} of a finite group $G$, denoted $\Phi(G)$, is defined as the intersection of all maximal subgroups of $G$. Clearly, $\Phi(G)$ is a proper, characteristic subgroup of $G$. Note that $\Phi(G)$ consists precisely of the elements $x$ of $G$ that are redundant in any generating set for $G$; that is, if $X$ is a generating set of $G$ containing $x$, then $X\setminus \{x\}$ is also a generating set of $G$.

For a finite $p$-group $G$ and an integer $j\geq 0$, the {\it $j$-th agemo subgroup\/} $\mho_j(G)$ is defined as
\[\mho_j(G)=\langle x^{p^j} \mid x\in G\rangle.\]
For example, if $p=2$ then $\mho_1(G)=\langle x^{2} \mid x\in G\rangle$. (The name ``agemo" is ``omega", read in reverse; and the symbol $\mho$ is an upside-down~$\Omega$.)  The agemo subgroups are related to the {\it omega subgroups} of a finite $p$-group $G$, which are defined by
$\Omega_j(G)=\langle x \in G \mid x^{p^j}=1\rangle$.
The agemo and omega subgroups are characteristic subgroups of $G$.
\smallskip

The following theorem is the well-known Burnside Basis Theorem for $p$-Groups (see \cite[Thm.\,1.12]{Ber2008}).

\begin{theorem}
\label{burnside}
Let $G$ be a finite $p$-group, and let $|G: \Phi(G)| = p^d$ with $d\geq 1$.\\[.04in]
(a)\, Then $G/\Phi(G) \cong \mathbb{Z}_p^d$, which is an elementary abelian $p$-group of rank $d$. Moreover, if $N \lhd G$ and $G/N$ is elementary abelian, then $\Phi(G) \leq N$. Thus $G/\Phi(G)$ is the largest elementary abelian quotient of $G$. \\[.02in]
(b)\, Every minimal generating set of $G$ contains exactly $d$ elements.\\[.02in]
(c)\, The commutator subgroup $G'$ and first agemo subgroups $\mho_1(G)$ of $G$ are subgroups of the Frattini group $\Phi(G)$ of~$G$, and the three subgroups are related by
\[\Phi(G)=G' \,\mho_1(G).\]
(d)\, If $p=2$ then $\Phi(G)=\mho_{1}(G)$.
\end{theorem}
\medskip

Next we study the implications of the Burnside Basis Theorem for string C-groups. To this end, let $(G, \{\rho_0, \rho_1, \cdots, \rho_{d-1}\})$ be a finite string C-group of 2-power order and Schl\"afli type $\{2^{s_1},2^{s_2},\ldots,2^{s_{d-1}}\}$, where $s_1,s_2,\ldots,s_{d-1} \geq 2$. Then $\rho_0, \rho_1, \cdots, \rho_{d-1}$ form a minimal generating set for $G$; in fact, this is true for any string C-group. Let
\begin{equation}
\label{agroup}
A:=G^{+}=\langle \rho_0\rho_1, \ldots, \rho_{d-2}\rho_{d-1}\rangle
\end{equation}
denote the rotation subgroup of $G$, which is of index 1 or 2 in $G$. Then $A$ also has 2-power order and the generators $\rho_0\rho_1, \ldots, \rho_{d-2}\rho_{d-1}$ form a minimal generating set of $A$. In fact, if any generator in this generating set was redundant, then $A$ could be generated by fewer than $d-1$ elements and thus $G=\langle\rho_0,A\rangle$ could be generated by fewer than $d$ elements. This is impossible, since all minimal generating sets for $G$ consist of $d$ elements because $G$ is a 2-group. As a consequence, since $A$ is a 2-group, any minimal generating set of $A$ will consist of $d-1$ elements.

Further, for $i=1,\ldots,d-1$, let $A_i$ denote the subgroup of $A$ given by
\begin{equation}
\label{aigroups}
A_i = \langle \rho_0\rho_1, \ldots, \rho_{i-2}\rho_{i-1},(\rho_{i-1}\rho_{i})^2, \rho_{i}\rho_{i+1},\ldots, \rho_{d-2}\rho_{d-1}\rangle.
\end{equation}
We will see later, these generators form a minimal generating set for $A_i$. Finally, define the two subgroups $B$ and $C$ of $A$ by
\begin{equation}
\label{bcgroup}
B:=\bigcap_{i=1}^{d-1} A_i,\quad
C:=\bigcap_{i=1}^{d-1} \Phi(A_i) .
\end{equation}
\smallskip

These subgroups of the rotation group can be exploited to establish the following general theorem about string C-groups of 2-power order. The fourth part of the theorem also establishes our main Theorem~\ref{mainthm1}.
\smallskip

\begin{theorem}
\label{stringC}
Let $d,n\geq 2$ and $G$ be a string C-group of order $2^n$ and Schl\"afli type $\{2^{s_1},2^{s_2},\ldots,2^{s_{d-1}}\}$, where $s_1,s_2,\ldots,s_{d-1} \geq 2$. Then,\\[.03in]
(a)\, $|\Phi(G)|= 2^{n-d}$ and $\Phi(G)=G'=\mho_1(G)$.\\[.02in]
(b)\, $G'=B$, with $B$ as in \eqref{bcgroup}.\\[.02in]
(c)\, $|C|=2^{n-2d+1}$, with $C$ as in \eqref{bcgroup}.\\[.02in]
(d)\, $G/C$ is a string C-group of order $2^{2d-1}$ and $G/C\cong G(\mathcal{C}_d)$, where $\mathcal{C}_d$ is the regular $d$-polytope of type $\{4,4,\ldots,4\}$ described in Section~\ref{secsmallest}.
\\[.02in]
(e)\, $C=G_3$, where $G_3:=[G,G,G]:=[[G,G],G]$.
\end{theorem}

\begin{proof}
The first part of the theorem can be settled using three applications of Theorem~\ref{burnside}. Suppose $G$ has distinguished involutory generators $\rho_0,\ldots,\rho_{d-1}$. The generators of a string C-group are known to be a minimal generating set, so Theorem~\ref{burnside}(b) directly implies that $|G: \Phi(G)| = p^d$ and thus $|\Phi(G)|=2^{n-d}$. Moreover, working modulo the commutator subgroup $G'$ and exploiting that $G$ is generated by $d$ involutions, shows immediately that $G/G'$ is an elementary abelian 2-group and thus $\Phi(G)\leq G'$, now by Theorem~\ref{burnside}(a). As $G$ is a $p$-group with $p=2$, Theorem~\ref{burnside}(c,d) now yields $\Phi(G)=G'=\mho_1(G)$. This settles the first part of the theorem.

We proceed with the second part of the theorem. First note that the rotation subgroup $A$ of $G$ must necessarily have index 2 in $G$. In fact, otherwise, $A=G$, which is impossible since then $G$ could be generated by fewer than $d$ elements, namely the $d-1$ generators $\rho_i\rho_{i+1}$ of $A$. It follows that $|A|=2^{n-1}$.

Next consider the subgroups $A_i$ of $A$ defined in \eqref{aigroups}. We claim that
$A_i \unlhd G$ and $|A_i|=2^{n-2}$ for each $i$. It is straightforward that $A_i \unlhd A$; in fact,  conjugating a generator of $A_i$ by a generator $\rho_{j-1}\rho_{j}$ of $A$ either leaves the generator of $A_i$ unchanged or explicitly gives a product of three generators of $A_i$. Similarly, $A_i$ is normalized by the generator $\rho_0$ of $G$; in fact, for $i\geq 4$ this is clear from the commutation rules in string $C$-groups, but for $i=1,3$ and $i=2$ observe additionally that
$\rho_0(\rho_1\rho_2)\rho_{0}= (\rho_0\rho_1)^{2}\rho_1\rho_2$ or
$\rho_0(\rho_1\rho_2)^{2}\rho_{0}= (\rho_0\rho_1)(\rho_1\rho_2)^{-2}(\rho_1\rho_0)$, respectively. Thus, $A_i \unlhd G$ since $G=\langle \rho_0, A\rangle$.

Moreover, $|A:A_i| \leq 2$. In fact, all generators of $A$, save $\rho_{i-1}\rho_i$, lie in $A_i$,  and so $A/A_i=\langle\rho_{i-1}\rho_{i}A_i\rangle$. Since $(\rho_{i-1}\rho_i)^2$ also lies in $A_i$, it follows that $|A/A_i|\leq 2$.

On the other hand, $A_i$ must be a proper subgroup of $A$. In fact, $(\rho_{i-1}\rho_{i})^2 \in \mho_1(A)$ and $\mho_1(A)=\Phi(A)$ by Theorem~\ref{burnside}(d). It follows that the element $(\rho_{i-1}\rho_{i})^2$ of $\Phi(A)$ is redundant in any generating set of $A$ that includes it. Now, if $A_{i}=A$, then the generators for $A_i$ in \eqref{aigroups} are also generators of $A$ and thus $(\rho_{i-1}\rho_{i})^2$ could be eliminated from this list of generators of $A$, leaving only the $d-2$ generators of the form $\rho_{j-1}\rho_{j}$, $j\neq i$, for $A$. This is impossible since all minimal generating sets of $A$ consist of $d-1$ elements. Thus, $|A:A_i| = 2$ and $|A_i|=2^{n-2}$.

We also require that the generators $\rho_0\rho_1, \ldots, \rho_{i-2}\rho_{i-1},(\rho_{i-1}\rho_{i})^2, \rho_{i}\rho_{i+1},\ldots, \rho_{d-2}\rho_{d-1}$ of $A_i$ in \eqref{aigroups} form a minimal generating set of $A_i$. Clearly, the generator $(\rho_{i-1}\rho_{i})^2$ cannot be redundant, since otherwise $A_i$ would lie in $A$, which is not true as we just saw. The argument is more complicated for the generators $\rho_{j-1}\rho_{j}$ of $A_i$ with $j\neq i$. Suppose that $\rho_{j-1}\rho_{j}$ can be eliminated from the generating set for $A_i$ for some $j\neq i$. We claim that then the subgroup
$\widehat{A}_{i}:=\langle \rho_{i-1},\rho_i,A_i\rangle$ of $G$ must coincide with $G$. First note that $A_{i}\unlhd \widehat{A}_{i}$ since $A_{i}\unlhd G$, and that
\[ \widehat{A}_{i}/A_i = \langle\rho_{i-1}A_i,\rho_{i}A_i \rangle.\]
Here, the two generators $\rho_{i-1}A_i,\rho_{i}A_i$ are nontrivial since $A_i$ lies in $A$ but
$\rho_{i-1}$ and $\rho_{i}$ do not. Moreover, their product is also nontrivial; otherwise,
$\rho_{i-1}\rho_{i}$ would lie in $A_i$ and thus $A_i$ would coincide with $A$. It follows that
$\widehat{A}_{i}/A_i\cong C_2\times C_2$ and importantly, that $\widehat{A}_{i}/A_i = G$ since $|A_i|=2^{n-2}$. Now we can complete the argument as follows. By the definitions of $\widehat{A}_{i}$ and the redundancy assumption of $\rho_{j-1}\rho_{j}$ for the generating set of $A_i$, the $d$ elements
\begin{equation}
\label{redundant}
\rho_{i-1},\rho_i, (\rho_{i-1}\rho_{i})^2,\, \rho_{k-1}\rho_{k}\;(k\neq i,j)
\end{equation}
form a generating set for the entire group $G$. But $(\rho_{i-1}\rho_{i})^2$ lies in $\mho_{1}(G)=\Phi(G)$, and so $(\rho_{i-1}\rho_{i})^2$ is a redundant element in any generating set of $G$ that contains it. Thus $(\rho_{i-1}\rho_{i})^2$ can be eliminated from the generating set in \eqref{redundant}, leaving a generating set of $G$ with fewer than $d$ elements. This is impossible.

Summarizing our considerations about the subgroups $A_i$ so far, we have established that $A_i \unlhd G$, $|A:A_{i}|=2$, $|A_i|=2^{n-2}$, and that the generating set of \eqref{aigroups} for $A_i$ is minimal.

Next we investigate a sequence of subgroups $B_1,\ldots,B_{d-2}$ of $A$ related to the subgroups $A_1,\ldots,A_{d-1}$. Define $B_1:=A_1 \cap A_2$ and $B_{i}:=B_{i-1}\cap A_{i+1}$ for $i=2,\ldots,d-2$. Then,
$B_{i}=A_1\cap \ldots \cap A_{i+1}$ for $i=1,\ldots,d-2$. In particular, $B_{d-2}=B$, with $B$ defined as in \eqref{bcgroup}. Our goal is to show that $|B_{i}|=2^{n-(i+2)}$ for $i=1,\ldots,d-2$.

Consider $B_1=A_1 \cap A_2$. First note that $A_1A_2=A$. In fact, $\rho_0\rho_1\in A_2$,
$\rho_1\rho_2\in A_1$, and all other generators $\rho_{i-1}\rho_{i}$ of $A$ with $i\geq 3$ lie in both $A_1$ and $A_2$. It follows that all generators of $A$ lie in $A_1A_2$. Thus $A_1A_2=A$. We now can compute the order of $B_1$ from
\[|B_1|=\frac{|A_1||A_2|}{|A_1A_2|}=\frac{2^{n-2} \cdot 2^{n-2}}{2^{n-1}}=2^{n-3}.\]

We can argue similarly for $B_2=B_1 \cap A_3$. In fact, since the generator $\rho_2\rho_3$ of $A$ missing from $A_3$ lies in $B_1$, we can conclude that $B_1A_3=A$. It follows that
\[|B_2|=\frac{|B_1||A_3|}{|B_1A_3|}=\frac{2^{n-3} \cdot 2^{n-2}}{2^{n-1}}=2^{n-4}.\]

Proceeding in this manner with the remaining subgroups $B_i$, we finally arrive at $B_{d-2}=B_{d-3} \cap A_{d-1}=B$. In this case the generator $\rho_{d-2}\rho_{d-1}$ of $A$ missing from $A_{d-1}$ lies in $B_{d-3}$, so that $B_{d-3}A_{d-1}=A$. This gives
\[|B|=|B_{d-2}|=\frac{|B_{d-3}||A_{d-1}|}{|B_{d-2}A_{d-1}|}=\frac{2^{n-(d-1)} \cdot 2^{n-2}}{2^{n-1}}=2^{n-d}.\]

We now can settle the second part of the theorem. First note that the elements $(\rho_0\rho_1)^2$, $(\rho_1\rho_2)^2$, $\cdots$, $(\rho_{d-2}\rho_{d-1})^2$ all lie in $B$, so $G/B$ is necessarily abelian since its generators $\rho_{0}B,\ldots,\rho_{d-1}B$ commute. It follows that $G' \leq B$. On the other hand, by the first part of the theorem, $G'=\Phi(G)$ and $|\Phi(G)|=2^{n-d}$. Comparing group orders we find that $|G'|=2^{n-d}=|B|$ and therefore $G'=B$.

For the third part of the theorem, we investigate the Frattini groups and agemo groups of the subgroups $A_i$ of $A$. First note that $\Phi(A_i) \unlhd G$ since $A_i$ is normal in $G$ and $\Phi(A_i)$ is a characteristic subgroup of $A_i$. By Theorem~\ref{burnside}(d), $\Phi(A)=\mho_1(A)$, $\Phi(A_i)=\mho_1(A_i)$ and therefore
\[ \langle (\rho_0\rho_1)^2,  (\rho_1\rho_2)^2,  (\rho_2\rho_3)^2,  \cdots, (\rho_{d-2}\rho_{d-1})^2\rangle \,\leq\, \mho_1(A)=\Phi(A)\]
and
\begin{equation}
\label{phiaigroups}
\langle (\rho_0\rho_1)^2, \ldots, (\rho_{i-2}\rho_{i-1})^2,(\rho_{i-1}\rho_{i})^4, (\rho_{i}\rho_{i+1})^2,\ldots, (\rho_{d-2}\rho_{d-1})^2\rangle
\,\leq\, \mho_1(A_i) = \Phi(A_i)
\end{equation}
for $i=1,\ldots,d-1$. Note that, by Theorem~\ref{burnside},
\[|\Phi(A)|=\frac{|A|}{2^{d-1}}=2^{n-d},\;\;\;
|\Phi(A_i)|=\frac{|A_i|}{2^{d-1}}=2^{n-d-1}\; (i=1,\ldots,d-1).\]
Recall here that the  minimal generating sets for the 2-groups $A$ and $A_i$ all consist of $d-1$ elements.

We now consider subgroups $C_1,\ldots,C_{d-2}$ of $\Phi(A)$ that are constructed from the Frattini groups $\Phi(A_1),\ldots,\Phi(A_{d-1})$ in the same way in which the subgroups $B_1,\ldots,B_{d-2}$ of $A$ were constructed from $A_1,\ldots,A_{d-1}$. In other words,
define $C_1:=\Phi(A_1) \cap \Phi(A_2)$ and $C_{i}:=C_{i-1}\cap \Phi(A_{i+1})$ for $i=2,\ldots,d-2$, so that $C_i:=\Phi(A_1) \cap \ldots\cap \Phi(A_{i+1})$ for $i=1,\ldots,d-2$. Then, $C_{d-2}=C$, with $C$ as in \eqref{bcgroup}. Observe that each $C_{i}$ is normal in the whole group $G$ since each subgroup $\Phi(A_j)$ participating in the intersection for $C_i$ is normal in $G$. We will show that $|C_i|=2^{n-d-(i+1)}$ for each $i=1,\ldots,d-2$, using a similar line of argument as above.

For the proof for $C_1=\Phi(A_1) \cap \Phi(A_2)$ we first verify that $\Phi(A_1)\Phi(A_2)=\Phi(A)$. In fact, from \eqref{phiaigroups} we know that $(\rho_0\rho_1)^{2}\in \Phi(A_2)$, $(\rho_1\rho_2)^{2}\in \Phi(A_1)$, and that all other generators $(\rho_{i-1}\rho_{i})^2$ of $\Phi(A)$ with $i\geq 3$ lie in both $\Phi(A_1)$ and $\Phi(A_2)$. Hence all generators of $\Phi(A)$ lie in $\Phi(A_1)\Phi(A_2)$ and therefore $\Phi(A_1)\Phi(A_2)=\Phi(A)$. Then,
\[|C_1|=\frac{|\Phi(A_1)||\Phi(A_2)|}{|\Phi(A)|}
=\frac{2^{n-d-1} \cdot 2^{n-d-1}}{2^{n-d}}=2^{n-d-2}.\]

For $C_2=C_1 \cap \Phi(A_3)$ we can argue similarly, noting that $C_{1}\Phi(A_3)=\Phi(A)$ because $(\rho_2\rho_3)^2 \in C_1$. Then,
\[|C_2|=\frac{|C_1||\Phi(A_3)|}{|\Phi(A)|}=\frac{2^{n-d-2} \cdot 2^{n-d-1}}{2^{n-d}}=2^{n-d-3}.\]

In much the same way we can proceed with the remaining subgroups $C_i$ until we arrive at
$C=C_{d-2}=C_{d-3} \cap \Phi(A_{d-1})$. Then $C_{d-3}\Phi(A_{d-1})=\Phi(A)$ since $(\rho_{d-2}\rho_{d-1})^2 \in C_{d-3}$. It follows that
\[|C|=|C_{d-2}|=\frac{|C_{d-3}||\Phi(A_{d-1})|}{|\Phi(A)|}=\frac{2^{n-d-(d-2)} \cdot 2^{n-d-1}}{2^{n-d}}=2^{n-d-(d-1)}=2^{n-2d+1}.\]
This establishes the third part of the theorem.

For the fourth part of the theorem, we need to show that the quotient $G/C$ is a string C-group of rank $d$. Since $G=\langle\rho_0,\ldots,\rho_{d-1}\rangle$ itself is a string C-group of rank $d$, the induced generators of $G/C$ are involutions (because $C$ lies in $A$) and satisfy the commutation relations consistent with a string diagram. It remains to establish the intersection property \eqref{intprop} for $G/C$. To this end, we establish the intersection property of $G/C$ indirectly, namely by explicitly identifying $G/C$ as the automorphism group of the regular $d$-polytope $\mathcal{C}_d$ of Section~\ref{phiaigroups}, which of course is a C-group.

We first show that
\begin{equation}
\label{comrhoC}
[(\rho_i\rho_{i+1})^2, \rho_{i+2}]\,\in\, C \quad (i=0,\ldots,d-3).
\end{equation}
Since $C:=\Phi(A_1) \cap \ldots\cap \Phi(A_{d-1})$, this is equivalent to establishing that the commutator $[(\rho_i\rho_{i+1})^2, \rho_{i+2}]$ lies in $\Phi(A_j)$ for each $j=1,\ldots,d-1$. We check this property by way of example. Then the remaining cases can be dealt with similarly.

Consider, for example, the commutator $[(\rho_0\rho_1)^2, \rho_2]$ obtained when $i=0$. We treat the cases $j=1$ and $j>1$ separately. When $j=1$ we want to show that $[(\rho_0\rho_1)^2, \rho_2] \in \Phi(A_1)$. First note that, since $\rho_0$ and $\rho_2$ commute, we have
\[ \begin{array}{rll}
[(\rho_0\rho_1)^2, \rho_2]
\,=\, [\rho_1\rho_0\rho_1, \rho_2]^{\rho_0}\!\!\!
&=\, [\rho_0, \rho_1\rho_2\rho_1]^{\rho_1\rho_0}&\\[.04in]
&=\, [\rho_0, (\rho_1\rho_2)^2]^{\rho_1\rho_0}&
\!\!\!=\, \big(((\rho_1\rho_2)^2)^{\rho_0}\,(\rho_1\rho_2)^2\big)^{\rho_1\rho_0}.
\end{array}\]
Now recall that $\Phi(A_1) \unlhd G$ and $(\rho_1\rho_2)^2 \in \Phi(A_1)$. Then any conjugate of $(\rho_1\rho_2)^2$ also lies in $\Phi(A_1)$, and so does the entire last term in the previous equation. It follows that $[(\rho_0\rho_1)^2, \rho_2]$ lies in $\Phi(A_1)$.
When $j>1$ the argument is simpler. In this case, $\Phi(A_j) \unlhd G$ and $(\rho_0\rho_1)^2 \in \Phi(A_j)$, so every conjugate of $(\rho_0\rho_1)^2$ also lies in $\Phi(A_j)$. Therefore,
\[ [(\rho_0\rho_1)^2, \rho_2]
=(\rho_0\rho_1)^2\,((\rho_0\rho_1)^2)^{\rho_2}  \in \Phi(A_j).\]
Thus $[(\rho_0\rho_1)^2, \rho_2]  \in C$, as desired.

Note that \eqref{comrhoC} also implies the more general statement that
\begin{equation}
\label{comrhoCj}
[(\rho_i\rho_{i+1})^2, \rho_{j}]\,\in\, C \quad (0\leq i<d-1;\, 0\leq j\leq d-1).
\end{equation}
In fact, by the commutation rules for string C-groups, this commutator is trivial if $j<i-1$ or $j>i+2$; and \eqref{comrhoC} covers the case $j=i+2$ of \eqref{comrhoCj}. For $j=i$ and $j=i+1$, we have
$[(\rho_i\rho_{i+1})^2, \rho_{j}]=(\rho_{i}\rho_{i+1})^4$, so then \eqref{comrhoCj} follows from the fact that $(\rho_{i}\rho_{i+1})^4\in C$, by \eqref{phiaigroups}. Finally, for $j=i-1$,
\[ [\rho_{i-1},(\rho_{i}\rho_{i+1})^2]
=[\rho_{i-1},\rho_{i}\rho_{i+1}\rho_{i}]
=[\rho_{i}\rho_{i-1}\rho_{i}, \rho_{i+1}]^{\rho_{i}}
=[(\rho_{i-1}\rho_{i})^2,\rho_{i+1}]^{\rho_{i-1}\rho_{i}}\,\in\, C\]
since $C$ is normal in $G$, and therefore also $[(\rho_{i}\rho_{i+1})^2,\rho_{i-1}]\in C$. This proves \eqref{comrhoCj}.

Next we investigate the quotient group $G/C$ of $G$ by its normal subgroup $C$. As the orders of $G$ and $C$ are $2^n$ and $2^{n-2d+1}$, respectively, it is clear that $G/C$ has the desired order, $2^{2d-1}$. We also
know from \eqref{phiaigroups} and \eqref{comrhoC} that the elements
\begin{equation}
\label{allrels}
(\rho_0\rho_1)^4, (\rho_1\rho_2)^4, \ldots, (\rho_{d-2}\rho_{d-1})^4,\,
 [(\rho_0\rho_1)^2, \rho_2], [(\rho_1\rho_2)^2, \rho_3], \ldots, [(\rho_{d-3}\rho_{d-2})^2, \rho_{d-1}]
 \end{equation}
all lie in $C$ and thus become trivial in $G/C$; the same holds more generally for all  elements of the form $[(\rho_i\rho_{i+1})^2,\rho_j]$ with $i=0,\ldots, d-2$ and $j=0,\ldots,d-1$. Then the presence of the first $d-1$ relators of the form $(\rho_{i}\rho_{i+1})^4$ in $C$ implies that $G/C$ is a quotient of the Coxeter group $[4,4,\ldots,4]$ with the diagram \eqref{Coxdia}; and the additional presence of the last $d-2$ relators of the form $[(\rho_i\rho_{i+1})^2, \rho_{i+2}]$ in $C$ (as well as the implied presence of all relators
$[(\rho_i\rho_{i+1})^2, \rho_{j}]$ in $C$) shows that $G/C$ is also a quotient of the automorphism group $G(\mathcal{C}_d)$ of the $d$-polytope $\mathcal{C}_d$ determined  by the relations implicit in \eqref{Coxdia} and \eqref{exrelscd}. On the other hand, the groups $G/C$ and $G(\mathcal{C}_d)$ have the same order. Thus,
\begin{equation}
\label{isogcd}
G/C \,\cong\, G(\mathcal{C}_d).
\end{equation}

Finally, the intersection property of $G/C$ follows from the intersection property of $G(\mathcal{C}_d)$. Thus $G/C$ is a string C-group of rank $d$ and the corresponding regular $d$-polytope is the polytope $\mathcal{C}_d$ described in the previous section. This settles the fourth part of the theorem.

For the final part of the theorem we need to show that $C=G_3$. First note that
\[ (\rho_{i+1}\rho_{i})^4 = [(\rho_i\rho_{i+1})^2, \rho_{i+1}] = [\rho_i, \rho_{i+1}, \rho_{i+1}] \in G_3\]
and
\[ [(\rho_i\rho_{i+1})^2, \rho_{i+2}] = [\rho_{i}, \rho_{i+1}, \rho_{i+2}] \in G_3,\]
so the elements of \eqref{allrels} not only lie in $C$ but also in $G_3$. Now let $N$ denote the normal closure of the elements from \eqref{allrels} in $G$. Then $N\leq C,G_3$, since $C$ and $G_3$ are normal in $G$. Moreover, $G/N\cong G(\mathcal{C}_d)$, by the definition of $G(\mathcal{C}_d)$; and $G(\mathcal{C}_d)\cong G/C$, by the fourth part of the theorem. Thus $C=N \leq G_3$.

It remains to show that $G_{3}\leq C$. We know from Section~\ref{secsmallest} that $G(\mathcal{C}_d)$ is of nilpotency class 2. Thus $G/N$ is of nilpotency class $2$. Interpreting this nilpotency class in terms of lower central series, we arrive at the following lower central series for $G/N$,
\[G/N\,\geq\, [G/N,G/N]\,\geq\, [[G/N,G/N],G/N] = \{1_{G/N}\} =\{N\}.\]
But $[G/N,G/N]=[G,G]N/N$ and $[[G,G]N/N,G/N]= [G,G,G]N/N$, so that the series becomes
\[G/N\,\geq\, [G,G]N/N\, \geq\,  [G,G,G]N/N = \{N\}.\]
Hence, $G_{3}=[G,G,G]\leq N=C$, as desired.

This completes the proof of the theorem.
\end{proof}
\medskip

In conclusion, it would be interesting to prove or disprove similar results for automorphism groups of other classes of symmetric polytopes of 2-power order, such as semi-regular, chiral, or two-orbit polytopes~\cite{Hu2010,HuSch2026+,MonSch2012,Pel2025}, or hypermaps or hypertopes ~\cite{FLW2016,JS1996}.
\bigskip

\subsection*{Acknowledgment}
We are grateful for the opportunity to discuss our research during Egon Schulte's visit at Beijing Jiaotong University and Shanxi Normal University in December 2025, hosted by Yan-Quan Feng and Dong-Dong Hou. Special thanks are due to Professor Feng for valuable discussions and support. The work was supported by the National Natural Science Foundation of China (12371022, 12271318, 12201371).

\end{document}